\documentclass[12pt,a4paper,titlepage]{article}
\usepackage{mathtext,amssymb,amsmath,latexsym,amsthm}

\usepackage[utf8x]{inputenc}
\usepackage{default}

\usepackage[T1,T2A]{fontenc}
\usepackage[russian]{babel}
\usepackage[dvips]{graphicx}

\theoremstyle{definition}
\newtheorem{defi}{Определение}
\theoremstyle{plain}
\newtheorem{stat}{Утверждение}
\newtheorem{theo}{Теорема}
\newtheorem{lemm}{Лемма}

\newtheorem*{lemma}{Лемма 7$'$}
\newtheorem*{explain}{Описание метода}
\theoremstyle{remark}
\newtheorem{rem}{Замечание}

\begin{document}
\include{title}
\voffset=-25mm

\begin{center}
  \Large \textbf{Quantitative generalizations 
of Niederreiter's result concerning continuants}\par\medskip
%  \textbf{╨Ъ╤А╨╛╤В╨║╨╛╨▓╨░ ╨Э.╨Р., ╨Ъ╨░╨╜ ╨Ш.╨Ф.}
  \end{center}
  \vskip+1.0cm
  \centerline{\bf I.D. Kan, N.A. Krotkova\footnote{research is supported by RFBR grant no. 09-01-00371a}
}
\vskip+1.0cm

The main part of the present paper is written in Russian.
Here we give a  brief summary in English.

We give certain generalization of Niederreiter's result from
\cite{Nie}
concerning famous Zaremba's conjecture on existence of rational numbers with bounded partial quotients.

Let $a, m, s \in \mathbb{Z}_+, a\ge 2$.
We consider a
sequence of integers $a_1,\ldots,a_n$  and the  continuant $\langle a_1,\ldots,a_n\rangle$.
Suppose that $${a_i< N, i=1,\ldots,n, N\in \mathbb{Z}_+}.$$
By $f(a^m, N)=f_m$ 
we define the number of such  
sequences
with elements  bounded by $N$ and such that $$ \langle a_1,\ldots,a_n\rangle = a^m.$$
Define  polynomials $P_s=P_s(\lambda)$ in the following way:
  $$
  P_s(\lambda)=\left\{
      \begin{aligned}
      &{\lambda}^{3} - s{\lambda}^2 - s\lambda - s,
      %=0
      \quad  &\text{if } s\equiv 0 (\mod \text{ }8), s\ge 6 \quad &\text{(1)} \\
      &{\lambda}^{3} - \left(s + 2\right) {\lambda}^2 + \left(s + 2\right) \lambda - \left(s-2\right),
      %=0
      \quad  &\text{if }       s \equiv2 (\mod \text{ }8), s\ge 6 \quad  &\text{(2)}\\
      &{\lambda}^{3} - \left(s+ 2\right) {\lambda}^2 + s\lambda + \left(s + 4\right),
      %=0
      \quad  &\text{if } s \equiv4 \left(\mod \text{ }8\right), s\ge 6 \quad &\text{(3)}\\
      &{\lambda}^{3} - s{\lambda}^2 - \left(s + 2\right) \lambda + \left(s + 2\right),
      %=0
      \quad  &\text{if } s \equiv 6 (\mod \text{ }8), s\ge 6 \quad &\text{(4)}\\
      &{\lambda}^2 - 4\lambda -4,
      %=0
      \quad &\text{if } s=4 \quad &\text{(5)}\\
      &\lambda - 2
      ,
      %=0
      \quad  &\text{if }  s=2 \quad &\text{(6)}\\
      &\lambda - s - 1,
      \quad &\text{if } s\equiv 1 (\mod \; 2), s\ge 3\quad &\text{(7)}.
      \end{aligned}
      \right.
  $$

  {\bf Theorem 1.}\,\,{\it 
For any given $a>1$ and $s>1$ for all $m$  one has
  $$
  f(a^m, a^s)\ge \lceil C_1(s)m^{\log_2 {\lambda}}\rceil,
  $$
Here  $C_1 (s)$ depends on  $s$ only and
  $$
  C_1\gg s^{-\log_2 (s+1) - 3},
  $$
  where  the constant in the sign $\gg $ is  absolute and  $\lambda$ 
is the largest roolt of the polynomial $P_s(\lambda).$
  %╨Я╤А╨╕ $m \rightarrow \infty$ ╨╕ $N= a^s, s\in \mathbb{N}$ %s=2k, k\in \mathbb{N}$
  % ╤Б╨┐╤А╨░╨▓╨╡╨┤╨╗╨╕╨▓╤Л ╤Б╨╗╨╡╨┤╤Г╤О╤Й╨╕╨╡ ╨╛╤Ж╨╡╨╜╨║╨╕ ╨┤╨╗╤П $f(a^m, a^s-1)$:
  %\begin{itemize}
  %\item[(\textit{i})]╨┐╤А╨╕ ╤З╨╡╤В╨╜╨╛╨╝ $s \qquad f(a^m, a^s-1)\ge C_1(s)m^{1+\log_{2}\lambda},\quad C_1(s)\gg %2^{-\log^{2}_{2} s(1 + \bar{\bar{o}}(1)) }
      %\min(\frac{1}{12},
  %    \min \left(\frac{1}{26},(s)^{-\log_{2} (s+1) - 2}\right)
      %)
   %   $,
    %  $\lambda$ - ╨╜╨░╨╕╨▒╨╛╨╗╤М╤И╨╕╨╣ ╨╕╨╖ ╨┤╨╡╨╣╤Б╤В╨▓╨╕╤В╨╡╨╗╤М╨╜╤Л╤Е ╨║╨╛╤А╨╜╨╡╨╣ ╨╝╨╜╨╛╨│╨╛╤З╨╗╨╡╨╜╨░
      %$$
     %
      %\end{aligned}
      %$$
      %\newpage}
    }

Here we note that  
for the largest roots 
${\lambda}_1, {\lambda}_2, {\lambda}_3, {\lambda}_4$ 
of the polynomials $(1)-(4)$ one has
      $$
      \begin{aligned}
      &{\lambda}_1=s + 1 +\frac{{\theta}_1}{s^2},\qquad \text{ where } {\theta}_1 \in (-1, 0);\\
      &{\lambda}_2=s + 1 -\frac{3}{s^2} + \frac{3}{s^3} + \frac{{\theta}_2}{s^4},\qquad \text{ where } {\theta}_2 \in (0, 3);\\
      &{\lambda}_3=s +1 -\frac{3}{s^2} + \frac{3}{s^3} -\frac{{\theta}_3}{s^4},\qquad \text{ where } {\theta}_3 \in (-9,-6);\\
      &{\lambda}_4=s + 1 +\frac{{\theta}_4}{s^2}, \qquad \text{ where } {\theta}_4 \in (-1, 0).
      \end{aligned}
      $$

  %\item[(\textit{ii})] ╨┐╤А╨╕ $s=2$
 % \item[(\textit{ii})] ╨┐╤А╨╕ ╨╜╨╡╤З╨╡╤В╨╜╨╛╨╝ $s\ge 3 \qquad f(a^m, a^s-1)\ge C_2(s)m^{\log_{2}(s+1)},\qquad C_2(s)\gg (s+1)^{-3-\log_{2} (s+1)}%(1 + \bar{\bar{o}}(1)) }
   %   $
  %\end{itemize}

 {\bf Theorem 2.}\,\,{\it
There exists $s_0\in\mathbb{N}$
such that for $N\ge a^{s_0}$ and for  $m$
large enough one has
  $$
  f(a^m,N)\ge \lceil C_2(N)m^{\log_{2}\lambda} \rceil, \qquad C_2( N)\gg 2^{-5\log^2_2\log_2 N},
  $$
where   $\lambda$
is the largest real root of
 -of the polynomial  $(3)$ with $s=\lfloor \log_a N  \rfloor.$
  }

  We note that for
   $s=6$  the polynomial (4) is a characteristic polynomial for the matrix $2A$ where
  $$
  A=\begin{pmatrix}
  2&1&1\\
  2&0&1\\
  1&1&1
  \end{pmatrix}
  .$$

  {\bf Theorem 3.}\,\,{\it 
  For $m$ large enough and  $s=6$ one has the following improvement of the bound
from Theorem 1:
  $$
  f(a^m, a^6)\ge \left\lceil Cm^{1 + \frac{1}{5}\log_{2}\mu}\right\rceil \asymp m^{\log_{2}{2\sqrt[5]{\mu}}},
  $$
  here $C>0, \mu$ is the largest eigenvalue of the matrix
  $$
  B={\begin{pmatrix}
  2&1&1\\
  2&0&1\\
  1&1&1\\
  \end{pmatrix}}^5 + {\begin{pmatrix}
  0&1&-1\\
  0&1&-1\\
  0&1&-1
  \end{pmatrix}}
  \eqno (8)
  $$
  }

  We note that for  $s=6$ one has
  $$
  2\sqrt[5]{\mu}-\lambda > 0,0000756,
  $$
where  $\lambda$ is the lagrest eigenvalue of the matrix $2A$,
and 
   $\mu$ is the largest eigenvalue of $B.$

  {\bf Theorem 4.}\,\,{\it For  $m\ge 8$ one has
  $f(3^m, 4)\ge \left\lceil\frac{m+1}{4}\right\rceil$.
  }

{\bf Theorem 5.}\,\,{\it 
Let  $k\in \mathbb{Z}_+, k\ge 2 $. Then
  $f(2^{2^k-1},3) \ge 2^{k}$
 }

 \newpage

%\author{Кроткова Наталья, Кан Игорь}

%\title {Количественные обобщения результатов Нидеррейтера о цепных дробях.}
%\date{}
%\maketitle

  \begin{center}
  \Large \textbf{Количественные обобщения результатов Нидеррейтера о цепных дробях.}\par\medskip
%  \textbf{Кроткова Н.А., Кан И.Д.}
  \end{center}
  \vskip+1.0cm
  \centerline{\bf И.Д. Кан, Н.А. Кроткова
\footnote{Работа выполнена при поддержке РФФИ, грант №
09-01-00371а}
  }
\vskip+1.0cm
\begin{small}
\centerline{\bf Аннотация}
  Пусть $d$ - натуральное число $\ge 2$. Гипотеза Зарембы гласит, что существует
  ${c\in\{1,2,\ldots,d-1\}}$, где $(c,d) =1 $ и разложение $\frac{c}{d}$ в цепную дробь имеет все
  неполные частные меньшие константы $N$. Предполагается, что $N=6$, для простых $d$ предполагается $N=4$.
  В 1986 году Нидеррейтер доказал справедливость гипотезы Зарембы с $N=4$ для $d$, являющихся степенями 2 и 3
  (см. \cite{Nie}), а также для $d$, являющихся степенями 5 с $N=5$.
  %Позднее были получены результаты о ее справедливости при $d$ имеющем иной специальный вид.
В настоящей работе будут получены количественные обобщения
результатов Нидеррейтера, а именно, оценка снизу количества
последовательностей c ограниченными элементами, континуант которых
является степенью произвольного натурального числа $a\ge 2$ .
Также будет показана неулучшаемость некоторых оценок в рамках
рассматриваемого метода.
\end{small}

  \section{Введение.}

  Пусть числа $c, d\in\mathbb N$ таковы, что $(c,d) =1, \text{ где } 1\le c \le d-1 $. Тогда $\frac{c}{d}$ представима в виде цепной дроби
  $$\frac{c}{d}=\cfrac{1}{a_1+\cfrac{1}{a_2 + \cfrac{1}{\ldots + \cfrac{1}{a_n}}}} = [0;a_1,\ldots,a_n]=[a_1,\ldots,a_n],
  $$
  где $a_1,\ldots,a_n$ - неполные частные (также называемые элементами цепной дроби) и $ a_i\in\mathbb N, i=1,\ldots,n. $
  %В дальнейшем рассматриваются только дроби $\frac{c}{d}$, удовлетворяющие условию $a_n\neq 1$.

  Знаменатель конечной цепной дроби $[a_1,\ldots,a_n]=[\textbf{u}]$ - это функция от последовательности
  $\textbf{u}=(a_1,\ldots,a_n)$. Это число называется  {\slshapeконтинуатом} $\textbf{u}$ и обозначается
  $\langle a_1,\ldots,a_n\rangle$ или $\langle \textbf{u}\rangle$. Для каждой последовательности
  $\textbf{u}\in{\mathbb N}^{n}$ определим $\textbf{u\_}=(a_2,\ldots,a_n)$ и
  $\textbf{u}^{\textbf{---}}=(a_1,\ldots,a_{n-1})$. Будем также писать $\{ \textbf{u} \}$
  для последовательности $(a_n,\ldots,a_1)$. Для пустого континуанта $\textbf{u}$ положим
  ${[\textbf{u}]=[\{\textbf{u}\}]=0, \langle \textbf{u} \rangle=1}$.

\begin{stat}{(см. [1])}
Пусть ${\textbf{u}=(u_1,\ldots,u_n), n\ge 2}$ - последовательность натуральных чисел. Тогда:
\begin{itemize}
\item[(1)]
$$ \langle\textbf{u}\rangle =\begin{vmatrix}
u_1 & 1 & 0 & \hdotsfor{1} & 0 \\ -1 & u_2 & 1 & 0 & \hdotsfor{1} \\ 0 & -1 & u_3 & 1 & \hdotsfor{1} \\
 \hdotsfor{5}  \\
0 & 0 & \ldots & -1 & u_n
 \end{vmatrix}
$$\newpage
\item[(2)] $\langle u_1,\ldots, u_n\rangle = \langle u_n,\ldots, u_1 \rangle$
~\item[(3)] если $u_2\neq 1$, то $\langle 1, u_2 - 1,u_3, \ldots, u_n\rangle = \langle u_2,\ldots, u_n\rangle$
\end{itemize}
\end{stat}

  Гипотеза Зарембы гласит, что для любого $d\ge 2$ существует дробь $\frac{c}{d}$,
  все неполные частные которой строго меньше $N$. Предполагается, что $N=6$. В 1986 году
  Нидеррейтер показал, что гипотеза верна для чисел $d$, являющихся степенями 2 и 3 с ограничением
  $N=4$, а также для чисел $d$ являющихся степенями 5 с ограничением $N=5$ (см. \cite{Nie}).
  В 2002 году справедливость была доказана для $d$ - степени 6 с $N=6$ \cite {jap1}, в 2005
  году был получен положительный результат для $d$, имеющего вид $7^{k\cdot 2^{n}}$, где
  $k=1, 3, 5, 7, 9, 11$ и $N=4$ \cite{jap2}.
  В настоящей статье будет получена оценка снизу количества последовательностей натуральных чисел $(a_1,\ldots,a_n)$ с ограниченными элементами и континуантом, равным $a^m, где a, m \in \mathbb N, a\ge 2, n - \text{ не фиксировано}$, из которой очевидно следует оценка снизу количества цепных дробей с ограниченными неполными частными и знаменателем, равным $a^m$. Общий результат сформулирован ниже.

  \section{Основные результаты.}

  Пусть $a, m, s \in \mathbb{N}, a\ge 2$. Будем рассматривать последовательности натуральных чисел $a_1,\ldots,a_n$ произвольной длины $n$, континуант которых $\langle a_1,\ldots,a_n\rangle$ равен
  %$a^m$
  некоторой $m$-ой степени числа $a$,
  причем для элементов последовательностей выполнено:\\${a_i< N, i=1,\ldots,n, N\in \mathbb{N}}$.

  Пусть $f(a^m, N)=f_m$ - количество последовательностей указанного вида, континуант которых равен $~a^m$.
  Определим также многочлен $P_s=P_s(\lambda)$ при различных значениях $s$:
  $$
  P_s(\lambda)=\left\{
      \begin{aligned}
      &{\lambda}^{3} - s{\lambda}^2 - s\lambda - s,
      %=0
      \quad  &\text{при } s\equiv 0 (\mod \text{ }8), s\ge 6 \quad &\text{(1)} \\
      &{\lambda}^{3} - \left(s + 2\right) {\lambda}^2 + \left(s + 2\right) \lambda - \left(s-2\right),
      %=0
      \quad  &\text{при }       s \equiv2 (\mod \text{ }8), s\ge 6 \quad  &\text{(2)}\\
      &{\lambda}^{3} - \left(s+ 2\right) {\lambda}^2 + s\lambda + \left(s + 4\right),
      %=0
      \quad  &\text{при } s \equiv4 \left(\mod \text{ }8\right), s\ge 6 \quad &\text{(3)}\\
      &{\lambda}^{3} - s{\lambda}^2 - \left(s + 2\right) \lambda + \left(s + 2\right),
      %=0
      \quad  &\text{при } s \equiv 6 (\mod \text{ }8), s\ge 6 \quad &\text{(4)}\\
      &{\lambda}^2 - 4\lambda -4,
      %=0
      \quad &\text{при } s=4 \quad &\text{(5)}\\
      &\lambda - 2
      ,
      %=0
      \quad  &\text{при }  s=2 \quad &\text{(6)}\\
      &\lambda - s - 1,
      \quad &\text{при } s\equiv 1 (\mod \; 2), s\ge 3\quad &\text{(7)}.
      \end{aligned}
      \right.
  $$

  \begin{theo}
  Для любых натуральных фиксированных $a$ и $s$, не равных 1, для всех достаточно больших $m$ справедлива оценка
  $$
  f(a^m, a^s)\ge \lceil C_1(s)m^{\log_2 {\lambda}}\rceil,
  $$
  растущая степенным образом по $m$, \\где положительное число $C_1 (s)$ зависит только от $s$,
  $$
  C_1\gg s^{-\log_2 (s+1) - 3},
  $$
  константа в знаке "$\gg $"$\;$абсолютная, $\lambda$ - наибольший из действительных корней многочлена $P_s(\lambda).$
  %При $m \rightarrow \infty$ и $N= a^s, s\in \mathbb{N}$ %s=2k, k\in \mathbb{N}$
  % справедливы следующие оценки для $f(a^m, a^s-1)$:
  %\begin{itemize}
  %\item[(\textit{i})]при четном $s \qquad f(a^m, a^s-1)\ge C_1(s)m^{1+\log_{2}\lambda},\quad C_1(s)\gg %2^{-\log^{2}_{2} s(1 + \bar{\bar{o}}(1)) }
      %\min(\frac{1}{12},
  %    \min \left(\frac{1}{26},(s)^{-\log_{2} (s+1) - 2}\right)
      %)
   %   $,
    %  $\lambda$ - наибольший из действительных корней многочлена
      %$$
     %
      %\end{aligned}
      %$$
      %\newpage
      \begin{rem}
      \begin{flushleft} Пусть ${\lambda}_1, {\lambda}_2, {\lambda}_3, {\lambda}_4$ - наибольшие из действительных корней многочленов $(1)-(4)$ соответственно. Тогда:\end{flushleft}
      $$
      \begin{aligned}
      &{\lambda}_1=s + 1 +\frac{{\theta}_1}{s^2},\qquad \text{ где } {\theta}_1 \in (-1, 0);\\
      &{\lambda}_2=s + 1 -\frac{3}{s^2} + \frac{3}{s^3} + \frac{{\theta}_2}{s^4},\qquad \text{ где } {\theta}_2 \in (0, 3);\\
      &{\lambda}_3=s +1 -\frac{3}{s^2} + \frac{3}{s^3} -\frac{{\theta}_3}{s^4},\qquad \text{ где } {\theta}_3 \in (-9,-6);\\
      &{\lambda}_4=s + 1 +\frac{{\theta}_4}{s^2}, \qquad \text{ где } {\theta}_4 \in (-1, 0).
      \end{aligned}
      $$
      \end{rem}

  %\item[(\textit{ii})] при $s=2$
 % \item[(\textit{ii})] при нечетном $s\ge 3 \qquad f(a^m, a^s-1)\ge C_2(s)m^{\log_{2}(s+1)},\qquad C_2(s)\gg (s+1)^{-3-\log_{2} (s+1)}%(1 + \bar{\bar{o}}(1)) }
   %   $
  %\end{itemize}
  \end{theo}

  \begin{theo}
  Существует $s_0\in\mathbb{N}$, такое что при $N\ge a^{s_0}$, для достаточно больших $m$ справедлива оценка (растущая по $m$):
  $$
  f(a^m,N)\ge \lceil C_2(N)m^{\log_{2}\lambda} \rceil, \qquad C_2( N)\gg 2^{-5\log^2_2\log_2 N},
  $$
  $\lambda$ - наибольший из действительных корней многочлена $(3)$ при $s=\lfloor \log_a N  \rfloor.$
  \end{theo}

  \begin{rem}
  При $s=6$ многочлен (4) является характеристическим для матрицы $2A$, где
  $$
  A=\begin{pmatrix}
  2&1&1\\
  2&0&1\\
  1&1&1
  \end{pmatrix}
  $$
  \end{rem}

  \begin{theo}
  При достаточно большом $m$ и $s=6$ теорема 1 допускает улучшение. Справедлива оценка:
  $$
  f(a^m, a^6)\ge \left\lceil Cm^{1 + \frac{1}{5}\log_{2}\mu}\right\rceil \asymp m^{\log_{2}{2\sqrt[5]{\mu}}},
  $$
  где $C>0, \mu$ - наибольшее из действительных собственных значений
  матрицы
  $$
  B={\begin{pmatrix}
  2&1&1\\
  2&0&1\\
  1&1&1\\
  \end{pmatrix}}^5 + {\begin{pmatrix}
  0&1&-1\\
  0&1&-1\\
  0&1&-1
  \end{pmatrix}}
  \eqno (8)
  $$
  \end{theo}

  \begin{rem}
  При $s=6$ выполняется следующее неравенство:
  $$
  2\sqrt[5]{\mu}-\lambda > 0,0000756,
  $$
  где $\lambda$ - наибольшее из действительных собственных значений матрицы $2A$,

   $\text{ }\mu$ - наибольшее из действительных собственных значений матрицы $B.$
  \end{rem}

  \begin{theo}
  При $m\ge 8$ справедлива оценка:
  $f(3^m, 4)\ge \left\lceil\frac{m+1}{4}\right\rceil$.
  \end{theo}

  \begin{theo}
  Пусть  $k\in \mathbb{N}, k\ge 2 $, тогда справедлива оценка:
  $f(2^{2^k-1},3) \ge 2^{k}$
 \end{theo}

 \begin{rem}
 Утверждения теорем остаются верными, если вместо последовательностей рассматривать цепные дроби с ограниченными элементами.
 \end{rem}

 \section{Описание метода, основные леммы.}

 \begin{lemm}\cite{hensley}
 Пусть $b\in \mathbb{N}, b>1$, $\textbf{u}=(u_1,\ldots,u_n)$  - последовательность натуральных чисел, причем $u_n>1$.
 Положим  \begin{align*} \textbf{w}=&(u_1,\ldots,u_{n-1},u_n-1,1, b-1, u_n,\ldots,u_1),\\ \textbf{w}\prime=&(u_1,\ldots,u_{n-1},u_n-1,u_n+1,u_{n-1},\ldots,u_1).\end{align*}. Тогда $\langle \textbf{w}\rangle=b{\langle \textbf{u}\rangle}^2, \langle \textbf{w}\prime \rangle={\langle \textbf{u}\rangle}^2.$
 \end{lemm}

% \begin{stat}{[1]}
 %Пусть $\textbf{u}=(u_1,\ldots,u_n)$ - последовательность натуральных чисел. Тогда:
 %\begin{itemize}
 %\item[(1)] $\langle u_1,\ldots, u_n\rangle = \langle u_n,\ldots, u_1 \rangle,$
 %\item[(2)] если $u_2\neq 1$, то $\langle 1, u_2 - 1,\ldots, u_n\rangle = \langle u_2,\ldots, u_n\rangle.$
 %\end{itemize}
 %\end{stat}

 \begin{lemm}
 Пусть задано $m>s$. Пусть $b=a^r, \text{ где } a, \:r\in \mathbb{N}_0, \:a\ge 2, \:r\le s,$ \\$r\equiv~m~(~mod ~2)$ и  $\:\textbf{u}=(u_1,\ldots,u_n)$ - последовательность натуральных чисел, такая что ${1\le u_i\le a^s -1, i=1,\ldots, n,\quad u_1, u_n\neq 1, a^s-1 }\:$ и $\:{\langle\textbf{u}\rangle=a^{\frac{m-r}{2}}}$.

 Тогда:
 \begin{itemize}
 \item[(1)] Для последовательностей
 $$
 \textbf{w}=(w_1,\ldots, w_J)=
 \left\{
 \begin{aligned}
 &(u_1,\ldots,u_n,b-1,1, u_n-1,u_{n-1},\ldots,u_1)\qquad &\text{при } r\neq 0,\\
 &(u_1,\ldots,u_{n-1},u_n-1,b-1,1, u_n,\ldots,u_1)\qquad &\text{при } r\neq 0,\\
 &(u_1,\ldots,u_{n-1},u_n+1,u_n-1,u_{n-1},\ldots,u_1)\qquad &\text{при } r=0,\\
 &(u_1,\ldots,u_{n-1},u_n-1,u_n+1,u_{n-1},\ldots,u_1)\qquad &\text{при } r=0,\\
 \end{aligned}
 \right.
 $$
 справедливо $\langle \textbf{w}\rangle=a^{m};$
 \item[(2)] для перечисленных последовательностей $\textbf{w}=(w_1,\ldots, w_J):$
 \item[$\bullet$] $J=2(n+r);$
 \item[$\bullet$] $ 1\le w_j\le a^s-1$ для $1\le j\le J$;
 \item[$\bullet$] $w_1,w_{J} \neq1, a^s-1;$
 \item[$\bullet$] все эти последовательности различны.
 \end{itemize}
 \end{lemm}
 \begin{proof} Все утверждения леммы 2 следуют из леммы 1, утверждения 1 и определения последовательностей.% Если лемма применяется к последовательности $\textbf{u}=(2)$, то в рассмотренных случаях это происходит при $b\neq1$ и при этом получившаяся последовательность $\textbf{w}=(2,b-1,1,1)$ заменяется на $\textbf{w}=(2,b-1,2)$, а последовательность $(1,1,b-1,2)$ - на $$
 \end{proof}

 \begin{lemm}
 Пусть последовательности $\textbf{w}_1$  и $\textbf{w}_2$, $\langle \textbf{w}_1\rangle=\langle\textbf{w}_2\rangle =a^m $ получены из $ \textbf{u}=(u_1,\ldots,u_n), \langle\textbf{u}\rangle=a^{m_1} $ и $ \textbf{v}=(v_1,\ldots,v_k), \langle\textbf{v}\rangle=a^{m_2}$ по лемме 2 с использованием $b_1=a^{r_1} $ и $b_2=a^{r_2} $  соответственно. Тогда, если $\textbf{u}$ и $\textbf{v}$ различны, то $\textbf{w}_1$  и $\textbf{w}_2$ также различны.
 \end{lemm}

 \begin{proof}
  Пусть последовательность $\textbf{u}$ удовлетворяет условию леммы 2,\\ последовательность $\textbf{w}_1$ получена из $\textbf{u}$. Тогда длина $\textbf{w}_1$ может быть равна $2n$ или $2n+2$. Аналогично, если $\textbf{v}$ удовлетворяет условию леммы 2, то длина $\textbf{w}_2$, полученной из нее, может быть равна $2k$ или $2k+2$.

  Рассмотрим 2 случая:

  1. Пусть $n=k$ (длины последовательностей совпадают). Тогда если $r_1=r_2=0$, то последовательности $\textbf{w}_1$ и $\textbf{w}_2$ могут иметь вид
  $$
  \textbf{w}_1=
 \left[
 \begin{aligned}
 %&(a_1,\ldots,a_n,b-1,1, a_n-1,\ldots,a_1)\qquad &\text{при } r\neq 0,\\
 %&(a_1,\ldots,a_n-1,b-1,1, a_n,\ldots,a_1)\qquad &\text{при } r\neq 0,\\
 &(u_1,\ldots,u_{n-1}, u_n+1,u_n-1,u_{n-1},\ldots,u_1),\\
 &(u_1,\ldots,u_{n-1},u_n-1,u_n+1,u_{n-1},\ldots,u_1);\\
 \end{aligned}
 \right.
  $$

  $$
  \textbf{w}_2=
 \left[
 \begin{aligned}
 %&(a_1,\ldots,a_n,b-1,1, a_n-1,\ldots,a_1)\qquad &\text{при } r\neq 0,\\
 %&(a_1,\ldots,a_n-1,b-1,1, a_n,\ldots,a_1)\qquad &\text{при } r\neq 0,\\
 &(v_1,\ldots,v_{n-1},v_n+1,v_n-1,v_{n-1},\ldots,v_1),\\
 &(v_1,\ldots,v_{n-1},v_n-1,v_n+1,v_{n-1}\ldots,v_1),\\
 \end{aligned}
 \right.
  $$
  Предположим, что $\textbf{w}_1=\textbf{w}_2$. Это невозможно, так как в этом случае, во-первых, ${(u_1,\ldots,u_{n-1})=(v_1,\ldots,v_{n-1})}$ и, во-вторых, либо ${u_n=v_n}$, \\либо ${u_n + 1=v_n-1}$ и ${u_n -1=v_n +1}$ одновременно.
  % \begin{align*}
  %&\text{ либо } \textbf{u}=\textbf{v},\\
  %&\text{ либо } (u_1,\ldots,u_n-1)=(v_1,\ldots,v_n+1) \text{ и } (u_1,\ldots,u_n+1)=(v_1,\ldots,v_n-1) \text{ %одновременно.}
  %\end{align*}
  \\
  Если же $r_1,r_2>0$, то
  $$
  \textbf{w}_1=
 \left[
 \begin{aligned}
 %&(a_1,\ldots,a_n,b-1,1, a_n-1,\ldots,a_1)\qquad &\text{при } r\neq 0,\\
 %&(a_1,\ldots,a_n-1,b-1,1, a_n,\ldots,a_1)\qquad &\text{при } r\neq 0,\\
 &(u_1,\ldots,u_n, a^{r_1}-1,1,u_n-1,u_{n-1},\ldots,u_1),\\
 &(u_1,\ldots,u_{n-1},u_n-1,1,a^{r_1}-1,u_n,\ldots,u_1);\\
 \end{aligned}
 \right.
  $$

  $$
  \textbf{w}_2=
 \left[
 \begin{aligned}
 %&(a_1,\ldots,a_n,b-1,1, a_n-1,\ldots,a_1)\qquad &\text{при } r\neq 0,\\
 %&(a_1,\ldots,a_n-1,b-1,1, a_n,\ldots,a_1)\qquad &\text{при } r\neq 0,\\
 &(v_1,\ldots,v_n, a^{r_2}-1, 1,v_n-1,v_{n-1},\ldots,v_1),\\
 &(v_1,\ldots,v_{n-1},v_n-1,1, a^{r_2}-1, v_n,\ldots,v_1).\\
 \end{aligned}
 \right.
 %\end{aligned}
  $$
  Если предположить, что $\textbf{w}_1=\textbf{w}_2$, получим, ввиду $\textbf{u}\neq \textbf{v}$,
  %\begin{align*}
  %&\text{либо } \textbf{u}=\textbf{v},\\
  %\text{либо }
  $$(u_1,\ldots,u_n,a^{r_1}-1,1,u_n-1,u_{n-1},\ldots,u_1) = (v_1,\ldots,v_{n-1},v_n-1,1, a^{r_2}-1, v_n,\ldots,v_1),$$
  %\end{align*}
  то есть
  %$$
  %\left\{
  %\begin{aligned}
  %&
  $u_n=v_n-1 \text{ и }
  %\\
  %&
  u_n - 1=v_n,
  %\\
  %\end{aligned}
  %\right.
  $ %$
  одновременно, что невозможно.

  2. Пусть $n=k+1$ (длины последовательностей $\textbf{u}$ и $\textbf{v}$ отличаются на 1), тогда $r_1=1, r_2=0$ и
  $$
  \textbf{w}_1=
 \left[
 \begin{aligned}
 %&(a_1,\ldots,a_n,b-1,1, a_n-1,\ldots,a_1)\qquad &\text{при } r\neq 0,\\
 %&(a_1,\ldots,a_n-1,b-1,1, a_n,\ldots,a_1)\qquad &\text{при } r\neq 0,\\
 &(u_1,\ldots,u_k, u_{k+1}+1,u_{k+1}-1,u_k,\ldots,u_1),\\
 &(u_1,\ldots,u_k,u_{k+1}-1,u_{k+1}+1,u_k,\ldots,u_1);\\
 \end{aligned}
 \right.
  $$

  $$
  \textbf{w}_2=
 \left[
 \begin{aligned}
 %&(a_1,\ldots,a_n,b-1,1, a_n-1,\ldots,a_1)\qquad &\text{при } r\neq 0,\\
 %&(a_1,\ldots,a_n-1,b-1,1, a_n,\ldots,a_1)\qquad &\text{при } r\neq 0,\\
 &(v_1,\ldots,v_k, a^{r_2}-1, 1,v_k-1,v_{k-1},\ldots,v_1),\\
 &(v_1,\ldots,v_{k-1},v_k-1,1, a^{r_2}-1, v_k,\ldots,v_1),\\
 \end{aligned}
 \right.
  $$
  предположив, что $\textbf{w}_1=\textbf{w}_2$, получим $v_{k}=u_k \text{ и } v_k-1=u_k$ одновременно, что невозможно.
  \end{proof}

  \begin{rem}
  Пусть в результате применения леммы 2 была получена последовательность
  $\textbf{u}=(u_1,\ldots,u_n), \; \langle\textbf{u}\rangle=a^m,\; 1\le u_i\le a^s-1,\; i=1,\ldots,n$,
  тогда для последовательностей ${(1, u_1-1,\ldots, u_n), (u_1,\ldots,u_n-1,1)}$ и ${(1, u_1-1,\ldots,u_n-1, 1)}$ выполнено:
  \begin{itemize}
  \item[1.] континуанты равны $a^m$;
  \item[2.] элементы являются натуральными числами;
  \item[3.] элементы строго ограничены сверху числом $N=a^s$.
  \item[4.] указанные последовательности различны.
  \end{itemize}
  %существуют, различны и имеют все элементы ограниченными $a^s-1$.
  \end{rem}
\newpage
  \begin{explain}
  Для оценки $f_m$, где $m\ge2$, будем использовать числа $g_m$, определенные следующим образом:
  $$
  g_2=\ldots=g_{s+2}=1;
  $$
  при $m\ge s+3$  числа $g_m$ определяются рекуррентным соотношением
  $$
  g_m=2\sum_{\substack{r=0\\ r\equiv m \:(\mod 2)}}^{s} g_{\frac{m-r}{2}}=2\sum_{k=\left\lfloor \frac{m-s+1}{2} \right\rfloor }^{\left\lfloor\frac{m}{2}\right\rfloor}g_{k}.
  \eqno (9)
  $$
  При $s\ge 3$ для $m=2,\ldots,s$ справедливо: $f_m\ge 4 g_m$, так как для каждого из таких $m$ будет существовать хотя бы одна последовательность натуральных чисел $(u_1,\ldots, u_n) \text{ c } u_i< a^s, i=1,\ldots,n$, у которой $u_1,u_n\neq1, a^s-1$ Это следует из существования для каждого такого $m$ несократимых дробей со знаменателями $a^m$, числителями меньше $\frac{a^m}{2}$, неполные частные представления в виде цепной дроби которых ограничены (строго) $a^s$. Множитель $4$ появляется согласно замечанию $5$. Последовательности, континуанты которых равны ${a^{s+1}, a^{s+2}\text{ и } a^{s+3}}$, будут существовать согласно лемме 2, примененной к последовательностям с континуантами $a^{\left\lfloor\frac{s}{2}\right\rfloor}, a^{\left\lfloor\frac{s-2}{2}\right\rfloor}$ и замечанию $5$.\\
  Для $m\ge s+3$, в результате применения лемм 2, 3 с различными ${r\in{1,\ldots, s}}$, \\где ${r=m \:(\mod 2)}$, получим, что ${f_m\ge g_m}$. Применение замечания $5$ дает ${f_m\ge4 g_m}$.\\
  %в результате применения леммы 2 с множитель 2 перед суммой $(8)$ появляется вследствие лемм 2 и 3 (для каждой $\textbf{u}$, удовлетворяющей условию леммы, получаем 2 последовательности $\textbf{w}$ и ${\textbf{w}\prime}$). В
  %итоге, после применения замечания к $\textbf{w}$ и $\textbf{w} \prime$:
  %$$
  %\text{для каждого } m\ge 2: \qquad f_m\ge4 g_m.
  %$$
  (Для случаев ограничений $N$, не рассмотренных здесь, рассуждения будут проведены отдельно.)

  Метод состоит в том, чтобы дать возможно более точную нижнюю оценку для $g_m$ на основании формулы $(9)$, примененной достаточно большое количество раз.
  \end{explain}

  \begin{theo}Пусть $a, s\in \mathbb{N}, a\ge 2.$
 В рамках рассматриваемого метода невозможно при фиксированном $s$ и $m\rightarrow\infty$ получить более точную (по сравнению с теоремой ~1) по порядку оценку снизу для величины $f(a^m, a^s)$ в случаях:
\begin{itemize}
 \item[\textit{(i)}] нечетного $s\ge 3$,
 \item[\textit{(ii)}] $s=4$,
 \item[\textit{(iii)}] $s=2$.
 \end{itemize}
 \end{theo}

  \section{Доказательства теорем.}
  \begin{defi}

  Для двух векторов $u \in\mathbb{N}^n_0 \text{ и } v\in \mathbb{N}^n_0 \text{ положим } u>v \; (u\ge v),$ $ \text{ если для каждого } i, 1\le i \le n \text{ выполнено } u_i > v_i\; (u_i\ge v_i). $

  Аналогично, для двух матриц $A=(a_{ij}) \in\mathbb{N}^{n^2}_0 \text{ и } B=(b_{ij})\in \mathbb{N}^{n^2}_0 \text{ положим }$ \\${ A>B \; (A\ge B),}\text{ если для каждой пары } i,j: 1\le i,j \le n \text{ выполнено }$
  \\${ a_{ij} > b_{ij}\; (a_{ij}\ge b_{ij}).} $
  \end{defi}
\newpage
  \begin{proof}[Доказательство теоремы 1]
  %\begin{itemize}
  %\item[(\textit{i})]
  $ \text{ }$\par
  \begin{flushleft}\textbf{Случай четного $\textbf{s}\ge \textbf{6}$}. \par Доказательство разобьем на 2 части: $s=2 (
 \mod 4) \text{ и } s=0 (\mod 4)$\end{flushleft}

  1. Пусть  $s=4q+2, q\in \mathbb{N}$.
  Для любого четного $m > s+2$ рекуррентная формула $9$ для $g_m$ содержит $\frac{s}{2}+1$ слагаемое, а для нечетного $m > s+2$ указанная сумма имеет $\frac{s}{2}$ слагаемых. Выделим среди чисел $g_m, m\ge 2$ классы $A_1\subset A_2\subset A_3$:
  $$
  \begin{aligned}
  &A_1=\{ g_l | l\equiv0(mod 2)\}\\
  &A_2=\{ g_l | l\equiv0(mod 2)\}\cup \{ g_l | l\equiv1(mod 4)\}\\
  &A_3=\{ g_l | l \in \mathbb{N}\}
  \end{aligned}
  \eqno (10)$$
  Таким образом, при $m\ge s+3$, принадлежность $g_m$ классу можно определить следующим образом:

  $A_1=\{ g_m | $ сумма в (9) содержит $ \frac{s}{2} + 1 $ слагаемых, из которых $\lceil\frac{s}{4}\rceil$ четные\},% из которых $ \lceil\frac{s}{2}\rceil $  четные$\}$

  $A_2=\{ g_m | $ сумма в (9) содержит хотя бы $ \frac{s}{2} $ слагаемых, из которых хотя бы $ \lceil\frac{s}{4} \rceil$  четные$\}$

  $A_3=\{ g_m | $ сумма в (9) содержит хотя бы $ \frac{s}{2} $ слагаемых, из которых хотя бы $\lfloor\frac{s}{4}\rfloor$ четные$\}$.
  %&A_2=\{ g_m | m\equiv0(mod 2)\}\\
  %&A_3=\{ g_m | m \in \mathbb{N}\}
  %\end{aligned}
  %$$

  Пусть дано достаточно большое $m$. Тогда результат одного шага оценки $g_m$ по рекуррентной формуле $(9)$ можно записать, используя скалярное умножение вектора-строчки на вектор-столбец, следующим образом:
  $$
  g_m\ge 2(q, \left\lfloor\frac{q + 1}{2}\right\rfloor, \left\lceil\frac{q + 1}{2}\right\rceil)\begin{pmatrix} a^0_1\\ \rule{0pt}{5mm}a^0_2\\ \rule{0pt}{5mm} a^0_3 \end{pmatrix},
  \eqno (11)$$
  где $a^0_1, a^0_2, a^0_3$ - наименьшие из элементов классов $A_1, A_2 \text{ и } A_3 $ соответственно, получившихся после применения одного шага рекуррентной формулы; $q, \left\lfloor\frac{q + 1}{2}\right\rfloor, \left\lceil\frac{q+1}{2}\right\rceil$ - количества элементов, заведомо принадлежащих каждому из классов, получившихся после применения формулы. Причем $\left\lceil\frac{q+1}{2}\right\rceil$ - количество элементов $A_2$, не учтенных при подсчете количества элементов $A_1$; $\left\lfloor\frac{q+1}{2}\right\rfloor$ - количество элементов $A_3 $, не учтенное при подсчете элементов $A_1$ и $A_2$.

  После первого применения формулы $(9)$, к каждому слагаемому, входящему в получившуюся сумму, также применим эту рекуррентную формулу, и т.д. Рассмотрим последовательность вложенных интервалов $I_0\subset I_1\subset\ldots\subset I_n$, соответствующих интервалам изменения индексов $k $ слагаемых $g_k$, входящих в сумму $(9)$ для $g_m$ после $j$~ применений рекуррентной формулы, где $j=1,\ldots,n+1$.
  $$
  \begin{aligned}
  I_0=&\left\{\; v\in\mathbb{N} \quad | \quad \lfloor\frac{m-s+1}{2}\rfloor\le v\le \lfloor\frac{m}{2}\rfloor \; \right\}, %\qquad
   &i_0=\min_{\substack{v \in \text{ } I_0}} v;\\
  I_1=&\left\{\; v\in\mathbb{N} \quad | \quad \lfloor\frac{i_0-s+1}{2}\rfloor\le v\le \lfloor\frac{m}{2}\rfloor \; \right\}, %\qquad
   &i_1=\min_{\substack{v \in \text{ } I_1}} v;\\
  %\qquad \qquad
  &\qquad \qquad
  \qquad \qquad
  \cdots \cdots \cdots &\\
  I_{n-1}= &\left\{\; v\in\mathbb{N} \quad | \quad \lfloor\frac{i_{n-2}-s+1}{2}\rfloor\le v\le \lfloor\frac{m}{2}\rfloor \; \right\}, %\qquad
   &\quad i_{n-1}=\min_{\substack{v \in \text{ }I_{n-1}}} v;\\
  I_{n}= &\left\{\; v\in\mathbb{N} \quad | \quad \lfloor\frac{i_{n-1}-s+1}{2}\rfloor\le v\le \lfloor\frac{m}{2}\rfloor \; \right\}, %\qquad
   &i_{n}=\min_{\substack{v \in \text{ }I_{n}}} v, \\
  \end{aligned}
  \eqno (12)
  $$
  Число $n=n(m)$ мы хотим выбрать как можно больше. Единственным ограничением является условие $i_{n-1}\ge s+3$ (иначе дальнейшее применение формулы $(9)$ незаконно).
  \\Таким образом, $i_n = \max \{ \; j \quad | \quad i_{j-1} \ge s+3\; \}. $

  \begin{flushleft} Определим также числа $a^i_j,\; i=1,\ldots,n; j=1,2,3$:%\end{flushleft}
  $$
  a^i_j=\min \{\; g_k\in A_i\quad | \quad k\in I_j\;\}.
  $$
  \end{flushleft}
  \begin{lemm} Для $l=0,\ldots,n-1:$
  $$
  \begin{pmatrix} a^l_1\\ \rule{0pt}{5mm}a^l_2\\\rule{0pt}{5mm} a^l_3 \end{pmatrix}\ge 2A\begin{pmatrix} a^{l+1}_1\\ \rule{0pt}{5mm}a^{l+1}_2\\ \rule{0pt}{5mm}a^{l+1}_3 \end{pmatrix}, \qquad \text{ где } A=\begin{pmatrix}q+1 & \left\lfloor\frac{q+1}{2}\right\rfloor & \left\lceil\frac{q+1}{2}\right\rceil\\ \rule{0pt}{5mm} q+1 & \left\lfloor \frac{q}{2} \right\rfloor & \left\lceil\frac{q}{2}\right\rceil \\ \rule{0pt}{5mm} q & \left\lfloor \frac{q+1}{2} \right\rfloor & \left\lceil\frac{q+1}{2}\right\rceil \end{pmatrix}
  \eqno (13)
  $$
  \end{lemm}
  \begin{proof}
  Действительно, для каждого $j\in\{1, 2, 3\}, \quad a^l_j\; $
  - это наименьший из элементов класса $A_j$, получившихся после $l+1$
  применений рекуррентной формулы (9). Продолжим оценку этого элемента по рекуррентной формуле.
  Подсчитаем, сколько индексов $k$ в сумме $(9)$ для $a^l_j$ гарантированно принадлежат классам
  $A_1$ и $A_2$ (остальные учтем как принадлежащие $A_3$).
  Число
  $a^l_1$ будет оцениваться снизу удвоенной суммой $q+1$ элемента класса $A_1$,  $\left\lceil\frac{q+1}{2}\right\rceil$ элементов класса $A_2$ и $\left\lfloor\frac{q+1}{2}\right\rfloor$ элементов класса $A_3$, причем индексы указанных элементов принадлежат $\left[i_{l+1}, i_{l+1}+\frac{s+2}{2}\right]$; указанное выражение не меньше, чем сумма, получающаяся при умножении первой строки матрицы $2A$ на столбец $(a^{l+1}_1, a^{l+1}_2, a^{l+1}_3)^{\perp}$, так как $a^{l+1}_j \text{ где } j=1,2,3$ - наименьшие из элементов соответствующих классов с индексами, принадлежащими $\left[i_{l+1}, i_{l+1} + \frac{s+2}{2}\right]$. Аналогично рассматривается $a^l_2$ и $a^l_3$ и, соответственно, вторая и третья строки матрицы $A$.
  \end{proof}

  \begin{lemm}
  Пусть $n$ - максимально возможно количество применений формулы (9)
  при вычислении $g_m$, ограниченное условием $i_{n-3}\ge s+3$, тогда
  $$n\ge\left\lceil\log_2 \left\lfloor\frac{m-s+1}{4s+4}\right\rfloor\right\rceil.
  \eqno (14)$$
  \end{lemm}
  \begin{proof}
  Заметим, что $i_j\le 2i_{j+1} + s$, поэтому
   $$i_0\le 2 i_{1} + s\le2(2i_{2} + s) + s\le\ldots\le2^{n}i_n + (2^n -1)s, \text{ где } i_n\in\{2,\ldots,s+2\}.$$
   Отсюда
   $$
   n\ge \min_{i_n\in{2,\ldots,s+2}} \log_2 \left(\frac{i_0 + s}{i_n + s} \right)=log_2 \left(\frac{\left\lfloor\frac{m-s+1}{2}\right\rfloor + s}{2s+2} \right).
   $$
   Таким образом, %количество шагов, которые мы гарантированно сможем произвести:
   $$
   n\ge\left\lceil \log_{2}\left\lfloor\frac{m-s+1}{4(s+1)}\right\rfloor\right\rceil.
   $$
  \end{proof}

  \begin{lemm} Пусть $\lambda$ - наибольшее из действительных собственных
  значений матрицы $2A$, где $A$ определена в (13), тогда:
  $$
  (q, \left\lfloor\frac{q+1}{2}\right\rfloor, \left\lceil\frac{q+1}{2}\right\rceil )(2A)^n\begin{pmatrix} a^n_1\\ a^n_2\\ a^n_3 \end{pmatrix}\ge
  %\frac{1}{3}\left\lfloor\frac{q}{2}\right\rfloor
  \frac{1}{3}{\lambda}^n.
  \eqno (15)$$
  \end{lemm}
  \begin{proof}
  Так как матрица $ A $ слева и справа домножается на вектора из положительного октанта, то
  $$
  (q, \left\lfloor\frac{q+1}{2}\right\rfloor, \left\lceil\frac{q+1}{2}\right\rceil)(2A)^n\begin{pmatrix} a^n_1\\ a^n_2\\ a^n_3 \end{pmatrix}\ge \max (2A)^n,
  \eqno (16)
  $$
  где $\max (2A)^n$ - максимальный элемент матрицы $(2A)^n$.

  Пусть $\bar{x}$ - собственный вектор матрицы $2A$, отвечающий собственному значению $\lambda$ и \\ ${| \:\bar{x}\:|=\max (x_j\; | \; j=1,2,3)}$, тогда
  $$
  %\begin{aligned}
  2A\bar{x}=\lambda \bar{x} , %\\
  %&\cdots\\
  \quad (2A)^n\bar{x}={\lambda}^n \bar{x}.
  %\end{aligned}
  $$
  Оценим $| \:(2A)^n\bar{x}\: |$:\\ c одной стороны
  $
  | \:(2A)^n\bar{x}\: |=| \:{\lambda}^n\bar{x}\: |={\lambda}^n| \:\bar{x}\: |
  $ так как $\lambda>0$;\\
  с другой стороны, для некоторого $i\in\{1, 2, 3\}$, \\
  ${|\:(2A)^n \bar{x} \: |= |\: a^{(n)}_{i1}x_1 + a^{(n)}_{i2}x_2 +a^{(n)}_{i3}x_3\:|\le 3\max (2A)^n|\:\bar{x}\:|}$, где ${(2A)^n=(a^{(n)}_{ij})}$.\\
  Поэтому, $$\max (2A)^n\ge \frac{1}{3}{\lambda}^n.
  \eqno (17)
  $$
  Объединяя неравенства (16) и (17), получаем
  $$
  (q, \left\lfloor\frac{q+1}{2}\right\rfloor, \left\lceil\frac{q+1}{2}\right\rceil)(2A)^n\begin{pmatrix} a^n_1\\ a^n_2\\ a^n_3 \end{pmatrix}\ge \frac{1}{3}
  %\left\lfloor\frac{q}{2}\right\rfloor
  {\lambda}^n.
  $$
  \end{proof}

  Применение лемм 4, 5 и 6 позволяет при $m\ge 5s+5$ продолжить неравенство (11):
  $$
  \begin{aligned}
  &g_m\ge
  %2(\left\lfloor\frac{q}{2}\right\rfloor, \left\lceil\frac{q}{2}\right\rceil, q)\begin{pmatrix} a^0_1\\ a^0_2\\ a^0_3 \end{pmatrix}\ge
  2(q, \left\lfloor\frac{q+1}{2}\right\rfloor, \left\lceil\frac{q+1}{2}\right\rceil)(2A)^n\begin{pmatrix} a^n_1\\ a^n_2\\ a^n_3 \end{pmatrix}
  %\ge2(\left\lfloor\frac{q}{2}\right\rfloor, \left\lceil\frac{q}{2}\right\rceil, q)(2A)^n\begin{pmatrix} 1\\ 1\\ 1 \end{pmatrix}\ge \frac{2}{3}\left\lfloor\frac{q}{2}\right\rfloor {(2\lambda)}^n
  \ge  \frac{2}{3}{(\lambda)}^{\left\lceil\log_2\left\lfloor\frac{m-s+1}{4s+4}\right\rfloor\right\rceil}
  \ge \\ &\ge\frac{2}{3}\frac{{\lambda}^{\log_2 m}}{{\lambda}^{\log_2 (s+1) + 3}}\gg \frac{1}{s^{\log_2 (s+1) + 3}}m^{\log_2 \lambda},
  \end{aligned}
  \eqno (18)$$
  где константа в знаке $\gg$ является абсолютной.
  %Здесь было использовано, что $\lambda\in\left(\frac{s}{2},\frac{s+1}{2}\right)$. Это является очевидным следствием следующей леммы.
  \begin{lemm}
  Наибольшее из действительных собственных собственных значений матрицы $2A$, где матрица $A$ определена в $(13)$, - это наибольший из действительных корней
 \item[$\textit{(i)}$] многочлена (2), при $s=8r + 2, l\in\mathbb{N}$;
 \item[$\textit{(ii)}$] многочлена (4), при $s=8r + 6, l\in\mathbb{N_0}$;
  \end{lemm}
  \begin{proof}
  В случаях $\textit(i)$ и $\textit(ii)$ матрица $A$ принимает, соответственно, вид
  $$
  \begin{pmatrix} 2r+1& r & r+1\\ 2r+1 &r & r \\ 2r& r& r+1\end{pmatrix}, \quad\qquad \begin{pmatrix} 2r+2& r+1 & r+1\\ 2r +2 & r & r+1 \\ 2r+1 & r+1 & r+1\end{pmatrix}.
  $$
  Таким образом, характеристическим для матрицы $2A$
  % в случае $\textit(i)$
  в случае $\textit{(i)}$ будет многочлен $(2)$, а в случае $\textit{(ii)}$ - многочлен $(4)$.
  % а в случае $\textit(ii)$ - многочлен
  % $-{\lambda}^3 + \frac{s}{2}{\lambda}^2 + \frac{s}{4}\lambda + \frac{s}{8}$, умножив его на $-1$, получим (1).\\
  %Во втором случае матрица $A$ примет вид
  %$$
  %A=\begin{pmatrix} r+1& r+1 & 2r+1\\ r & r+1 & 2r+2 \\ r & r+1 & 2r+1\end{pmatrix}.
  %$$
  %Характеристический многочлен этой матрицы имеет вид \\$-{\lambda}^3 + \left(\frac{s}{2}+1\right){\lambda}^2 - %\frac{s}{4}\lambda - \left(\frac{s}{8}+\frac{1}{2}\right)$, умножив его на $-1$, получим (3).
  %
  \end{proof}

  2. Пусть $s=4q,\text{ где } q\in \mathbb{N}, q\ge 2$

  Заметим, что для любого четного $m > s+2$ рекуррентная формула $(9)$ для $g_m$ содержит $\frac{s}{2}+1$ слагаемое, а для любого нечетного $m > s+2$ указанная сумма имеет $\frac{s}{2}$ слагаемых. Выделим среди чисел $g_m, m\ge 2$ классы $A_1\subset A_2\subset A_3$:
  $$
  \begin{aligned}
  &A_1=\{ g_l | l\equiv0\:(\mod 4)\}\\
  &A_2=\{ g_l | l\equiv0\:(\mod 2)\}\\
  &A_3=\{ g_l | l \in \mathbb{N}\}
  \end{aligned}
  \eqno(19)
  $$
  Таким образом, при $m\ge s+3$, принадлежность $g_m$ классу можно определить следующим образом:

  $A_1=\{ g_m | $ сумма в $(9)$ имеет $ \frac{s}{2} + 1 $ слагаемое, из которых $ \lceil\frac{s+2}{4}\rceil $  четные$\}$

  $A_2=\{ g_m | $ сумма в $(9)$ имеет $ \frac{s}{2} + 1 $ слагаемое, из которых хотя бы $ \lfloor\frac{s+2}{4} \rfloor$  четные$\}$

  $A_3=\{ g_m | $ сумма в $(9)$ имеет хотя бы $ \frac{s}{2} $ слагаемых, из которых хотя бы  $\lfloor\frac{s}{4}\rfloor$ четные$\}$
  \begin{flushleft}Проводя рассуждения, аналогичные тем, что были проведены при рассмотрении предыдущего случая, получаем вид матрицы $ A $ и для каждого $m\ge 5s+5$ оценку $g_m$:\end{flushleft}
  $$
  A=\begin{pmatrix} \left\lfloor\frac{q+1}{2}\right\rfloor & \left\lceil\frac{q+1}{2}\right\rceil & q\\ \rule {0pt}{5mm}\left\lfloor \frac{q}{2} \right\rfloor & \left\lceil\frac{q}{2}\right\rceil & q+1 \\ \rule{0pt}{5mm}\left\lfloor \frac{q}{2} \right\rfloor & \left\lceil\frac{q}{2}\right\rceil & q \end{pmatrix},
  \eqno(20)$$

  $$
  \begin{aligned}
  &g_m\ge2(\left\lfloor\frac{q}{2}\right\rfloor, \left\lceil\frac{q}{2}\right\rceil, q){(2A)}^n\begin{pmatrix} 1\\ 1 \\ 1 \end{pmatrix}\ge %\frac{2}{3}
  %\left\lfloor\frac{q+1}{2}\right\rfloor
  %{\lambda}^{\left\lceil\log_2 \left\lfloor\frac{m-s+1}{4s+4}\right\rfloor\right\rceil}\ge
   \frac{2}{3}{\lambda}^{\left\lceil \log_2 {\left\lfloor\frac{m-s+1}{4s+4}\right\rfloor} \right\rceil}
   \ge \\ &\ge\frac{2}{3}
  %\left\lfloor\frac{q+1}{2}\right\rfloor
  \frac{{\lambda}^{\log_2 m}}{{\lambda}^{\log_2 (s+1) + 3}}\gg \frac{1}{s^{\log_2 (s+1) + 3}}m^{\log_2 \lambda}
  \end{aligned}
  $$
  %Здесь,так же как и в первой части доказательства, использовано, что %$\lambda\in\left(\frac{s}{2},\frac{s+1}{2}\right)$.

  \begin{lemma}
  Наибольшее из действительных собственных собственных значений матрицы $2A$, где матрица $A$ определена в $(20)$, - это наибольший из действительных корней
 \item[$\textit(i)$] многочлена (1), при $s=8r, l\in\mathbb{N}$;
 \item[$\textit(ii)$] многочлена (3), при $s=8r+4, l\in\mathbb{N}$;
  \end{lemma}
  \begin{proof}
  В случаях $\textit(i)$ и $\textit(ii)$ матрица $A$ принимает вид
  $$
  \begin{pmatrix} r& r+1 & 2r\\ r &r & 2r+1 \\ 2r& r& r+1\end{pmatrix}, \quad\qquad \begin{pmatrix} r+1& r+1 & 2r+1\\ r & r+1 & 2r+2 \\ r & r+1 & 2r+1\end{pmatrix}.
  $$
  Тогда характеристическим для матрицы $2A$ будет многочлен $(1)$ или $(3)$ соответственно.

  %\\ $-{\lambda}^3 + \left(\frac{s}{2}+ 1\right){\lambda}^2 - \left(\frac{s}{4}+\frac{1}{2}\right)\lambda + \left(\frac{s}{8}-\frac{1}{4}\right)$, умножив его на $-1$, получим (2).\\
  %Во втором случае матрица $A$ примет вид
  %$$
  %A=\begin{pmatrix} 2r+2& r+1 & r+1\\ 2r +2 & r & r+1 \\ 2r+1 & r+1 & r+1\end{pmatrix}.
  %$$
  %Характеристический многочлен этой матрицы имеет вид $$-{\lambda}^3 + \frac{s}{2}{\lambda}^2 + %\left(\frac{s}{4}+\frac{1}{2}\right)\lambda - \left(\frac{s}{8}+\frac{1}{4}\right),$$умножив его на $-1$, получим %(4).
  \end{proof}
\newpage
  \begin{flushleft}\textbf{Случай $\textbf{s=4}$.} \end{flushleft}\par
  Необходимо среди чисел $g_m, m\ge 2$ выделить классы
  $$
  \begin{aligned}
  &A_1=\{\:g_m\; | \; m\equiv 0 \:(\mod 2)\:\}\\
  &A_2=\{\:g_m\; | \; m\ge 2\}.
  \end{aligned}
  \eqno(21)
  $$
  Тогда, как и в случае четного $s\ge6$, введем матрицу
  $$
  A=\begin{pmatrix} 1& 2\\ 1& 1
  \end{pmatrix}
  \eqno (22)$$
  Пусть $\lambda$ - наибольшее из собственных значений матрицы $2A$.\\
  Снова рассмотрим интервалы  ${ I_0\subset\ldots\subset I_n}$, числа $i_0,\ldots, i_n$ (12) \\и числа ${a^j_i=\min \{g_k\in A_i\; | \;k\in I_j\}}$, тогда для каждого $m>s+2$

  $$
  g_m\ge 2\begin{pmatrix}1&1\end{pmatrix}\begin{pmatrix}a^0_1\\ a^0_2\end{pmatrix}.
  \eqno (23)
  $$
  если $m\ge 25$, то продолжим неравенство (23):
  $$
  \begin{aligned}
  g_m\ge
  %&2\begin{pmatrix}1&1\end{pmatrix}\begin{pmatrix}a^0_1\\ a^0_2\end{pmatrix}
  2\begin{pmatrix}1&1\end{pmatrix}{(2A)}^n\begin{pmatrix}1\\1\end{pmatrix}\ge
  %\begin{pmatrix}1&1\end{pmatrix}
  {\lambda}^{\left\lceil log_2 {\left\lfloor\frac{m-3}{20}\right\rfloor}\right\rceil}
  %\begin{pmatrix}1\\1\end{pmatrix}
  \gg m^{\log_2 \lambda}
  %\frac{1}{5}{(2(1+\sqrt{2}))}^{\log_2 m}\ge \frac{1}{5} m^{1 + \log_2 (1+\sqrt{2})}.
  \end{aligned}$$

  \begin{flushleft}\textbf{Случай $\textbf{s=2}$.} \end{flushleft}%В этом случае при $a=2$ невозможно применить рекуррентную формулу так, как это происходило в предыдущих случаях: не существует последовательностей $\textbf{u}=(u_1,\ldots,u_n)$, удовлетворяющих условию теоремы, \\таких что ${u_1,u_n\neq 1, 2^2 - 1 }$, континуанты которых равны $2^2$ и $2^5$.

  %При $a,s=2$ примем ${g_1=g_3=g_4=g_6=g_7=\ldots=g_{11}=1}$.\\Тогда ${f_i\ge g_i, \text{ при }}$ ${i=1, 3, 4, 6, 7,\ldots, 11}$, так как существуют последовательности, удовлетворяющие указанным выше условиям, континуанты которых равны соответственно ${2^1, 2^3, 2^4, 2^6, 2^7,\ldots, 2^{11}}$.
  Пусть $a=2$. Для $m=6, 7, \ldots, 11$ положим $g_6=\ldots=g_{11}=1$, так как существуют последовательности, удовлетворяющие условию нашей теоремы, действительно:
  $$
  \begin{aligned}
  %2^1&=\langle2\rangle,\\
  %2^3&=\langle2, 1, 2\rangle,\\
  %2^4&=\langle2, 3, 2\rangle,\\
  2^6&=\langle2, 1, 3, 1, 1, 2\rangle,\\
  2^7&=\langle2, 1, 2, 1, 1, 1, 1, 2\rangle,\\
  2^8&=\langle2, 3, 3, 1, 3, 2\rangle,\\
  2^9&=\langle2, 3, 2, 1, 1, 1, 3, 2\rangle,\\
  2^{10}&=\langle2, 3, 2, 3, 1, 1, 3, 2\rangle,\\
  2^{11}&=\langle2, 1, 1, 2, 3, 3, 1, 1, 2, 2\rangle.\\
  \end{aligned}
  $$
  Для $m\ge 12$ значения $g_m$ вычисляются по рекуррентной формуле (9), которая теперь принимает вид
  %Для четного $m$ рекуррентная формула для $g_m$ содержит 2 слагаемых, для нечетного $m$ - одно слагаемое. Поэтому матрица $A$ имеет вид:
  $$
  g_m \ge 2g_{\left\lfloor\frac{m}{2}\right\rfloor}.
  %A=
  %\begin{pmatrix}
  %1&1\\
  %0&1
  %\end{pmatrix}
  \eqno(24)$$
  %Собственные значения матрицы $A$ равны 1.\\
  %Формула для необходимого числа применений рекурсивной $n$ должна измениться, так как
  В этом случае, в обозначениях (12), ${i_n\in\{6, 7, 8, 9, 10, 11\}}$. Оценим снизу максимальное число $n=n(m)$ применений рекуррентной формулы $(9)$ для вычисления $g_m$. Используя такие же рассуждения, как и при доказательстве леммы $5$, получаем:
  $$
  {m\ge 2^n(i_n +1),}
  $$
  откуда
  $$
  n\ge \left\lceil\log_2 \frac{m+1}{12} \right\rceil.
  $$

  %В этом случае, в обозначениях (10), ${i_n\in\{6, 7, 8, 9, 10, 11\}}$. Тогда для вычисления $g_m$, количество применений
  %$$
  %n=\left\lfloor\log_2 \left(\frac{\left\lfloor\frac{m-s+1}{2}\right\rfloor + s}{2s+2} %\right)\right\rfloor==\left\lfloor\log_2\left(\frac{\left\lfloor\frac{m-1}{2}\right\rfloor + %2}{13}\right)\right\rfloor.
  %$$
  Тогда для $m\ge 12$ неравенство (24) можно продолжить:
  $$
  g_m\ge 2^n g_{i_n}\ge 2^{\left\lceil\log_2 \frac{m+1}{12} \right\rceil}\ge \frac{m+1}{12}.
  %  2\left(0, 1\right){\begin{pmatrix}1&1\\0 & 1\end{pmatrix}}^n 2^n \begin{pmatrix}1\\1\end{pmatrix}=
  %2\left(0, 1\right)\begin{pmatrix}1&n\\0& 1\end{pmatrix} 2^n \begin{pmatrix}1\\1\end{pmatrix}\ge 2^{n+1}\ge %\frac{1}{26} m.
  $$
  Если $a>2$, положим $g_2=g_3=g_4=1$ (так как существуют последовательности, удовлетворяющие условию теоремы, континуанты которых равны, соответственно, ${a^2, a^3, a^4}$ с неполными частными, ограниченными $a^2$). Для  ${m\ge 5}$ значения $g_m$ вычисляются по рекуррентной формуле (24). Максимальное количество применений рекуррентной формулы ${n=\left\lceil\log_2 \frac{m+1}{5}\right\rceil}$. Таким образом, при $m\ge 5$ получаем ${g_m\ge \frac{m+1}{5}}$.

   %то оценку $g_m$ можно проводить не внося изменений метод, то есть принять $g_1=\ldots=g_{4}=1$, а для всех $m\ge 5$ значения $g_m$ вычислять по формуле (8). Однако, если мы примем $g_1=g_2=\ldots=g_{11}=1$ и будем применять рекуррентную формулу для $m\ge 12$, то при $m\rightarrow\infty$ получим отличие только на константу, порядок роста не изменится.
  \par\medskip

  \begin{flushleft}\textbf{Случай нечетного $\textbf{s}\ge\textbf{3}$.} \end{flushleft}
  \par В рекуррентную формулу для ${g_m, \text{ где } m>s+2}$ входит $\frac{s+1}{2}$ слагаемое, не зависимо от четности $m$. Таким образом, для ${m>5s+5}$ справедлива следующая оценка:
  $$
  \begin{aligned}
  &g_m=2\left(g_{\left\lfloor\frac{m}{2}\right\rfloor} + \ldots + g_{\left\lfloor\frac{m-s+1}{2}\right\rfloor}\right)\ge2\frac{s+1}{2}\min_{v\in \left[i_0, i_0 + \frac{s+1}{2}\right]} g_v\ge (s+1)^2\min_{v\in \left[i_1, i_1 + \frac{s+1}{2}\right] } g_v \ge\ldots\ge\\&\ge(s+1)^n\ge(s+1)^{\left\lceil\log_2 \left\lfloor\frac{m-s+1}{4s+4}\right\rfloor\right\rceil}\ge
  \frac{(s+1)^{log_2 m}}{(s+1)^{3+log_2 (s+1)}}=\frac{m^{log_2 (s+1)}}{(s+1)^{3+log_2 (s+1)}}.
  \end{aligned}
  $$
  При оценке было использовано, что интервалы $I_0\subset I_1\subset\ldots\subset I_n$ изменения индексов слагаемых, входящих в рекуррентную формулу для $g_m$ соответствуют (12) и ${i_k, k\in{1, 2,\ldots, n}}$. Также было использовано, что $n$ - максимальное количество применений формулы (9) - может быть оценено согласно лемме 5.
  %\par\medskip
  %Таким образом, теорема 1 полностью доказана.
  %\end{proof}
  %Таким образом, если $i_0\neq i_n$, то неравенство (7) можно продолжить:
  %$$
  %\ge  2(\lfloor\frac{q}{2}\rfloor, \lceil\frac{q}{2}\rceil, q)(2A)^n\begin{pmatrix} a^n_1\\ a^n_2\\ a^n_3 %\end{pmatrix}\ge
  %\end{itemize}

  Теорема 1 полностью доказана.
  \end{proof}

  \begin{lemm}[\textbf{О сравнении корней многочленов (1)-(4)\text{ и } (7)}]
  Среди наибольших действительных корней многочленов (1)-(4) и (7) наименьшим при
  %$s\rightarrow\infty$
  всех достаточно больших $s$ является корень многочлена (3).
  \end{lemm}
  \begin{proof}
  Пусть ${\lambda}_1, {\lambda}_2, {\lambda}_3, {\lambda}_4, {\lambda}_7$ - наибольшие из корней многочленов (1), (2),(3),(4) и (7) соответственно. Тогда ${\lambda}_i \in\left(s, s+1\right)$, при $i\in\{1, 2, 3, 4\}$, то есть ${\lambda}_i=\underline{\underline{O}} (s)$. Следовательно, для ${\lambda}_1$:
  $$
  \begin{aligned}
  &{{\lambda}_1}^3=s{{\lambda}_1}^2 +s{\lambda}_1 + s,\\
  &{\lambda}_1=s +\frac{s}{{\lambda}_1} + \bar{\bar{o}}(1), \text{ получаем }\\
  &{\lambda}_1=s + 1 + \varepsilon,\qquad \text{ где } \varepsilon=\underline{\underline{O}} \left(\frac{1}{s}\right).
  \end{aligned}
  $$
  Поставляя это выражение в многочлен (1) и приравнивая к нулю, получаем
  $$
  \begin{aligned}
  &{\left(s + 1 +\varepsilon\right)}^{3} - s{\left(s + 1 +\varepsilon\right)}^2 - s\left(s + 1 +\varepsilon\right) - s=0, \text{ или, эквивалентно, }\\
  &{\varepsilon}^3 + (2s+3){\varepsilon}^2 + (s^2 + 3s + 3)\varepsilon +1=0, \text{ откуда }\\
  &\varepsilon=-\frac{1}{(s^2 + 3s + 3)} - \underbrace{\frac{(2s+3){\varepsilon}^2}{(s^2 + 3s + 3)} - \frac{{\varepsilon}^3}{(s^2 + 3s + 3)}}_{\underline{\underline{O}}\left(\frac{1}{s^3}\right)}=-\frac{1}{s^2} + \underline{\underline{O}}\left(\frac{1}{s^3}\right).
  \end{aligned}
  $$
  Следовательно,
  $$
  {\lambda}_1=s + 1 -\frac{1}{s^2} + \underline{\underline{O}}\left(\frac{1}{s^3}\right)=s + 1 + \frac{{\theta}_1}{s^2}, $$
  где ${\theta}_1\in(-1, 0).$

  Проводя аналогичные рассуждения для ${\lambda}_2, {\lambda}_3, {\lambda}_4$
  %${\lambda}_2$ - наибольшего из действительных корней многочлена (2),\\
  %${\lambda}_3$ - наибольшего из действительных корней многочлена (3),\\
  %${\lambda}_4$ - наибольшего из действительных корней многочлена (4),\\
  получаем:
  $$
  \begin{aligned}
  {\lambda}_2=s + 1 -\frac{3}{s^2} + \underline{\underline{O}}\left(\frac{1}{s^3}\right),\\
  {\lambda}_3=s +1 -\frac{3}{s^2} + \underline{\underline{O}}\left(\frac{1}{s^3}\right),\\
  {\lambda}_4=s + 1 -\frac{1}{s^2} + \underline{\underline{O}}\left(\frac{1}{s^3}\right)=s + 1 +\frac{{\theta}_4}{s^2}.
  \end{aligned}
  $$
  где ${\theta}_4\in(-1, 0)$

  Таким образом, необходимо сравнить при $s\rightarrow\infty$ величины ${\lambda}_2$ и ${\lambda}_3$.\\ Для этого положим
  $$
  \begin{aligned}
  &{\lambda}_2=s + 1 -\frac{3}{s^2} + {\varphi}_2,\\
  &{\lambda}_3=s + 1 -\frac{3}{s^2} + {\varphi}_3, \text{ где } {\varphi}_2, {\varphi}_3=\underline{\underline{O}}\left(\frac{1}{s^3}\right),
  \end{aligned}
  $$
  и подставим значения ${\lambda}_2,\text{ и } {\lambda}_3$ в соответствующие многочлены.
  Получим
  $$
  \begin{aligned}
  &{\varphi}_2=\frac{3}{s^3} + \frac{3}{s^4} + \underline{\underline{O}}\left(\frac{1}{s^5}\right)=\frac{3}{s^3} + \frac{{\theta}_2}{s^4}, \text{ где } {\theta}_2 \in(0, 3);\\
  &{\varphi}_3=\frac{3}{s^3} -\frac{6}{s^4} + \underline{\underline{O}}\left(\frac{1}{s^5}\right)=\frac{3}{s^3} -\frac{{\theta}_3}{s^4}, \text{ где } {\theta}_3 \in (-9, -6).\
  \end{aligned}
  $$
  Таким образом, учитывая что для каждого ${i\in\{1, 2, 3, 4\}}$ значение ${\lambda}_i<s+1$, получаем, что наименьшим из ${{\lambda}_1, {\lambda}_2,{\lambda}_3,{\lambda}_4 \text{} {\lambda}_7}$ при $s\rightarrow\infty $ является
  $$
  {\lambda}_3=s +1 -\frac{3}{s^2} + \frac{3}{s^3} -\frac{{\theta}_3}{s^4}, \qquad {\theta}_3 \in (-9, -6).
  $$
  \begin{rem}
  Более точное выражение для ${\lambda}_3$:
  $$
  \qquad {\lambda}_3=s +1 -\frac{3}{s^2} + \frac{3}{s^3} -\frac{6}{s^4} - \frac{9}{s^5} - \frac{15}{s^6} + \underline{\underline{O}}\left(\frac{1}{s^7}\right).
  \eqno(25)
  $$
  \end{rem}
  \end{proof}

  %Таким образом, теорема 1 полностью доказана.
  %\end{proof}

  \begin{proof}[Доказательство теоремы 2]
  $\text{ }$\par
  Согласно лемме 8, найдется $s_0\in\mathbb{N}$, такое что для всех $s\ge s_0$ наименьшим среди чисел ${{\lambda}_1, {\lambda}_2, {\lambda}_3, {\lambda}_4, s+1}$ является ${\lambda}_3$. Рассмотри произвольное ${N\in\mathbb{N}, N\ge a^{s_0}}$, и применим теорему 1 для ${s=\left\lfloor\log_2 N\right\rfloor}$.
  %   b{N}, \text{ где } N\ge a^3 - 1}$, \\положим ${s=\lfloor \log_a (N + 1) \rfloor}$ и применим теорему 1. Так как любого $s\ge 3$ при $m\rightarrow\infty$ выполнено:
  %\begin{itemize}
  %\item[ 1.]наименьшим среди ${{\lambda}_1, {\lambda}_2,{\lambda}_3,{\lambda}_4 }$ является ${\lambda}_3$,
  %\item[ 2.]${1+\log_2 \:{\lambda_3} \le \log_2 (s+1)}$
  %\item[ 3.]при ${s=4, \quad {\lambda_3}< 1+\sqrt{2}}$,
  %\end{itemize}
  Тогда для достаточно больших $m$ получим:
  $$
  f_m\ge 4g_m\ge C_3 (N) m^{log_2 \lambda_3},
  $$
  где $ C_3 (N)$ имеет вид:
  $$
  C_3\ge(s+1)^{-3-\log_{2} (s+1)}\ge \frac{1}{2^{5\log_2 {\log_2 N}}}.
  $$
  \end{proof}

  \begin{proof}[Доказательство теоремы 3]
  $\text{ }$\par
  Как и при доказательстве теоремы 1 выделим среди чисел $g_m, m\ge 1$ классы $A_1\subset A_2\subset A_3$ (10).
  Матрица $A$ будет иметь вид:
  $$
  A=\begin{pmatrix}
  2&1&1\\2&0&1\\1&1&1
  \end{pmatrix}
  \eqno(26)$$
  В основе доказательства лежит
  \begin{lemm}
  Для достаточно большого $m$ можно утверждать, что хотя бы на одном из 4-х последовательных шагов вычисления $g_m$ по рекуррентной формуле (9) встретилось слагаемое, индекс которого сравним с 5 по модулю 8.
  \end{lemm}
  \begin{proof}[]
  Доказательство проходит перебором всевозможных остатков при делении $m$ на 128.
  Для примера рассмотрим "наихудший" ${ }$ случай, то есть когда слагаемое,  индекс которого сравним с 5 по модулю 8, появляется на 4-м шаге: пусть ${m=83 + 128k, k\in\mathbb{N}}$, тогда
  $$
  \begin{aligned}
  &g_{83+128k}=2\left(g_{41+64k}+g_{40+64k}+g_{39+64k}\right)=\\
  &=4\left(2g_{20+32k}+3g_{19+32k}+3g_{18+32k}+2g_{17+32k}\right)=\\
  &=8\left(2g_{10+16k}+8g_{9+16k}+10g_{8+16k}+10g_{7+16k}+5g_{6+16k}\right)=\\
  &=16\left(2g_{5+8k}+20g_{4+8k}+35g_{3+8k}+35g_{2+8k}+25g_{1+8k} +5g_{8k}\right),
  \end{aligned}
  $$
  \end{proof}
  и интересующее нас слагаемое  - первое в полученном выражении.

  %Обозначим $g_k\in A_i\searrow A_j$, если число $g_k$ учтено как член класса $A_i$, но не учтено, как член класса %$A_j$.

  \begin{lemm}
  Пусть для достаточно большого $m$ при вычислении $g_m$ по формуле (9), появилось слагаемое с индексом, сравнимым с 5 по модулю 8. Тогда при следующем применении формулы (9) появится слагаемое, заведомо принадлежащее классу $A_2$, но не учтенное, как принадлежащее этому классу (то есть учтенное как число из $A_3$).%, при вычислениях с помощью матрицы $A$.
  \end{lemm}
  \begin{proof}
  Если $l=5 \:(\mod 8)$, то $l$ - нечетно. Следовательно, $g_l$ могло быть учтено как член класса $A_2$ или $A_3$. Если оно было учтено как член $A_2$, то на следующем шаге мы получим 2 элемента из класса $A_1$ и один из $A_2$ вместо 2 элементов из $A_1$ и одного из $A_3$. Если же $g_l$ учтено как элемент класса $A_3$, то получаем 2 элемента из $A_1$ и один из $A_2$ вместо одного из $A_1$, одного из $A_2$ и одного из $A_3$. Заметив, что $A_1\subset A_2$, получим утверждение леммы.
  %Пусть $g_l$ - это число, индекс которого сравним с 5 по модулю 8. Рассмотрим класс числа, "давшего" при применении формулы (8) слагаемое $g_l$.

  %1. Пусть это класс $A_1$. Тогда $g_l\in A_2\searrow A_1$,  вместе с $g_l$ появятся 2 слагаемых класса $A_1$ и одно слагаемое класса $A_3\searrow A_2$, применяя формулу еще раз, и принимая во внимание, что $l\equiv5 (\mod 8)$, получим:
  % \\7 слагаемых класса $A_1$, 4 слагаемых $A_2\searrow A_1$ и 3 слагаемых $A_3\searrow A_2$.\\
  % Если же не учитывать, что $l\equiv5 (\mod 8)$, то получится:\\
  % 7 слагаемых класса $A_1$, 3 слагаемых $A_2\searrow A_1$ и 4 слагаемых $A_3\searrow A_2$.

  %2. Пусть это класс $A_2$. Тогда $g_l\in A_3\searrow A_2$. Приняв во внимание, что $l=5 (\mod 8)$, при следующем применении формулы (8) получим:\\
  % 6 слагаемых класса $A_1$, 3 слагаемых $A_2\searrow A_1$ и 2 слагаемых $A_3\searrow A_2$;\\
  % Если не учитывать, что $l\equiv5 (\mod 8)$, то получится:\\
  % 5 слагаемых класса $A_1$, 3 слагаемых $A_2\searrow A_1$ и 3 слагаемых $A_3\searrow A_2$.\\
  % Так как $A_1\subset A_2$, то утверждение верно.

  %3. Пусть это класс $A_3$. Тогда $g_l\in A_2$. Приняв во внимание, что $l=5 (\mod 8)$, при следующем применении (8) получим:\\
  % 5 слагаемых класса $A_1$, 3 слагаемых $A_2\searrow A_1$ и 2 слагаемых $A_3\searrow A_2$;\\
  % Если не учитывать, что $l\equiv5 (\mod 8)$, то получится:\\
  % 5 слагаемых класса $A_1$, 2 слагаемых $A_2\searrow A_1$ и 3 слагаемых $A_3\searrow A_2$.
  \end{proof}

  %\begin{rem}
  %Сформулируем полученный в лемме результат в терминах матриц. Пусть на $i-$ом шаге вычисления $g_m$ возникло слагаемое $g_l, \text{ где } l\equiv5 \:(\mod 8)$. Согласно лемме 10, на следующем шаге появится слагаемое, заведомо лежащее в классе $A_2$, вместо слагаемого из $A_3\searrow A_2$. Таким образом, в обозначениях теоремы 1, это можно записать:
  %$$
  %\begin{pmatrix}a^{0}_1\\ \rule{0pt}{5mm}a^{0}_2\\ \rule{0pt}{5mm}a^{0}_3\end{pmatrix}\ge
  %2^{i+1}\left({\begin{pmatrix}2&1&1\\ \rule{0pt}{5mm}2&0&1\\ \rule{0pt}{5mm}1&1&1 \end{pmatrix}}^{i+1}+\begin{pmatrix}0&1&-1\\ %\rule{0pt}{5mm}0&1&-1\\ \rule{0pt}{5mm}0&1&-1 \end{pmatrix}\right)\begin{pmatrix}a^{i}_1\\ \rule{0pt}{5mm}a^{i}_2\\ %\rule{0pt}{5mm}a^{i}_3\end{pmatrix}.
  %\eqno(27)$$
  %\end{rem}

  \begin{lemm}
  Для достаточно большого $m$, при $0\le i\le n-5-5t$ ($t\in\mathbb{N}_0$, $n$ определено в ~(12)) верно:
  %после 8 применений формулы (8) при вычислении ${(a^{i-1}_1, a^{i-1}_2, a^{i-1}_3)}^{\perp}, i\ge 1$:
  $$
  %\begin{pmatrix}x_1& x_2& x_3\end{pmatrix}
  \begin{pmatrix}
  a^{i}_1\\ \rule{0pt}{5mm}a^{i}_2\\ \rule{0pt}{5mm}a^{i}_3
  \end{pmatrix}
  \ge {2^{5}B}^t
  %\begin{pmatrix}x_1& x_2& x_3\end{pmatrix}
  %\left({\begin{pmatrix}2&1&1\\ \rule{0pt}{5mm}2&0&1\\ \rule{0pt}{5mm}1&1&1 \end{pmatrix}}^{5}+\begin{pmatrix}0&1&-1\\ \rule{0pt}{5mm}0&1&-1\\ %\rule{0pt}{5mm}0&1&-1 \end{pmatrix}\right)%\begin{pmatrix}a^{i}_1\\a^{i}_2\\a^{i}_3\end{pmatrix}
  \left(2A\right)^5
  \begin{pmatrix}
  a^{i+5t+5}_1\\ \rule{0pt}{5mm}a^{i+5t+5}_2\\ \rule{0pt}{5mm}a^{i+5t+5}_3
  \end{pmatrix},
  \eqno(27)$$
  где
  $$
  %\text{где }
  B=A^5 + \begin{pmatrix}0 &1 & -1\\ \rule{0pt}{5mm}0 & 1 & -1\\ \rule{0pt}{5mm}0 & 1 & -1\end{pmatrix}$$
  % - матрица, не меньшая чем $\bar{A}$,
  %произойдет умножение ${(a^{i-9}_1, a^{i-9}_2, a^{i-9}_3)}^{\perp}$ на матрицу не меньшую, чем $B$, где
  %$$
  %\bar{A}=2^8\begin{pmatrix}
  %12958&5655&7276\\
  %10072&4396&5656\\
  %10087&4415&5682
  %\end{pmatrix}.
  %$$
  \end{lemm}

  \begin{proof}
  %Обозначим:
  %$$\begin{pmatrix}
  %0&1&-1
  %\end{pmatrix}=
  %\begin{pmatrix}
  %0&1&-1\\ \rule{0pt}{5mm}0&1&-1\\ \rule{0pt}{5mm}0&1&-1
  %\end{pmatrix}
  %$$
  При $t=0$ утверждение леммы следует из леммы 4, примененной 5 раз. Пусть лемма уже доказана для некоторого $t\ge 0$.

  Сформулируем полученные выше результаты в терминах матриц. Согласно леммам 9 и 10, для любого $i$ найдется $j\in\{2, 3, 4, 5\}$, такое что на $i+j$-ом шаге применения  формулы $(8)$ появится слагаемое, принадлежащее классу $A_2$, учтенное как элемент из $A_3$ (в терминах леммы 4). Другими словами,
  $$
  \begin{pmatrix}
  a^{i}_1\\ \rule{0pt}{5mm}a^{i}_2\\ \rule{0pt}{5mm}a^{i}_3
  \end{pmatrix}
  \ge 2^j \left({\begin{pmatrix} 2 & 1 & 1\\ \rule{0pt}{5mm} 2 & 0 & 1\\ \rule{0pt}{5mm} 1 & 1 & 1 \end{pmatrix}}^j + \begin{pmatrix} 0 & 1 & -1\\ \rule{0pt}{5mm} 0 & 1 & -1\\ \rule{0pt}{5mm} 0 & 1 & -1\end{pmatrix}\right)
  \begin{pmatrix}  a^{i+j}_1\\ \rule{0pt}{5mm}a^{i+j}_2\\ \rule{0pt}{5mm}a^{i+j}_3 \end{pmatrix},
  \eqno(28)
  $$
  Возьмем наименьшее из таких $j$.
  Для оценки вектора из правой части еще $5-j$ раз применим лемму 4, ввиду чего
  $$
  \begin{pmatrix}
  a^{i}_1\\ \rule{0pt}{5mm}a^{i}_2\\ \rule{0pt}{5mm}a^{i}_3
  \end{pmatrix}
  \ge 2^j \left({A}^j + \begin{pmatrix} 0 & 1 & -1\\ \rule{0pt}{5mm} 0 & 1 & -1\\ \rule{0pt}{5mm} 0 & 1 & -1\end{pmatrix}\right) (2A)^{5-j} \begin{pmatrix} a^{i+5}_1\\ \rule{0pt}{5mm}a^{i+5}_2\\ \rule{0pt}{5mm}a^{i+5}_3\end{pmatrix},
  \eqno(29)
  $$
  Докажем, что лемма верна для $t+1$. Применяя индуктивное предположение к вектору из правой части $(29)$, получим:
  $$
  \begin{pmatrix}
  a^{i}_1\\ \rule{0pt}{5mm}a^{i}_2\\ \rule{0pt}{5mm}a^{i}_3
  \end{pmatrix}
  \ge 2^5 \left({A}^5 + \begin{pmatrix} 0 & 1 & -1\\ \rule{0pt}{5mm} 0 & 1 & -1\\ \rule{0pt}{5mm} 0 & 1 & -1\end{pmatrix}{A}^{5-j}\right)
  (2^{5}B)^{t}(2A)^{5}\begin{pmatrix} a^{i+5t + 10}_1\\ \rule{0pt}{5mm}a^{i+5t + 10}_2\\ \rule{0pt}{5mm}a^{i+5t + 10}_3\end{pmatrix},
  \eqno(30)
  $$
  Лемма будет доказана, если показать, что минимум по $j\in\{2, 3, 4, 5\}$ произведения матриц в $(30)$ (в смысле определения 1) будет достигнут при $j=5$. Для доказательства этого факта оценим снизу произведение
  $$
  \begin{pmatrix} 0 & 1 & -1\\ \rule{0pt}{5mm} 0 & 1 & -1\\ \rule{0pt}{5mm} 0 & 1 & -1\end{pmatrix} {A}^{5-j}
  (2^{5} B)^{t}(2A)^{5}
  $$
  Для этого положим
  $$
  C=\left\{
  \begin{aligned}
  &B,\qquad &\text{при } t\ge 1,\\
  &A^5, \qquad &\text{при } t=0
  \end{aligned}
  \right.
  $$
  и проверим конечным перебором по $j\in\{2, 3, 4, 5\}$, что
  $$
  \begin{pmatrix} 0 & 1 & -1\\ \rule{0pt}{5mm} 0 & 1 & -1\\ \rule{0pt}{5mm} 0 & 1 & -1\end{pmatrix}{A}^{5-j}C\ge
  \begin{pmatrix} 0 & 1 & -1\\ \rule{0pt}{5mm} 0 & 1 & -1\\ \rule{0pt}{5mm} 0 & 1 & -1\end{pmatrix}C.
  $$
  Таким образом, минимум достигается при $j=5$.
  \par
  \begin{flushleft}
  Лемма доказана.
  \end{flushleft}
  \end{proof}

 Пусть $\lambda$ - наибольшее из действительных собственных значений матрицы ${2A}$, где $A$ определено в (26);
 $\mu$ - наибольшее из действительных собственных значений матрицы $B$ (8).
 Тогда ${2\sqrt[5]{\mu}- \lambda \ge 0.0000756}$ (что может быть проверено непосредственно на компьютере).
 %$\qquad\lambda(\bar{A})$ - наибольшее  действительное собственное значение матрицы ${\bar{A}}$.

 %Тогда $\lambda(\bar{A})\ge (\lambda(A))^7$.
 %\end{lemm}
 %\begin{proof}
 % Действительно, характеристический многочлен матрицы $\bar{a}$:
 % $$
 % -{\lambda}^3 + 23036{\lambda}^2 - 247384\lambda + 277224=0.
 % \eqno (23)$$
 % Из формулы (19): $\lambda(A)< \frac{167}{54}$. Возводя эту дробь в 7 степень и подставляя в (21),\smallskip %\begin{flushleft}
 %  получаем ${\lambda(\bar{A})> {(\frac{167}{54})}^8}$. Следовательно, $\lambda(\bar{A})\ge (\lambda(A))^8$.
 %  %\end{flushleft}
 %  \end{proof}
  %Обозначим за $\mu$ наибольшее из собственных значений матрицы $A^7 + \begin{pmatrix}0&1&0\end{pmatrix}$.
  %Непосредственное вычисление показывает, что
  Таким образом, используя леммы 9-11, для достаточно больших $m$ можно дать более точную оценку для $g_m$, чем та, что получена в теореме 1:
  $$
  g_m\gg m^{1+\frac{1}{7}\log_{2}\mu}.
  $$

  \end{proof}

  \begin{proof}[Доказательство теоремы 4.]
  $\text{ }$\par
  Будем рассматривать последовательности $(a_1,\ldots,a_n)$ произвольной длины $n$, для которых
  \begin{itemize}
  \item[$1.$] континуант $\langle a_1,\ldots,a_n\rangle=3^m$,где $m\in \mathbb{N}$,
  \item[$2.$] $a_i\le 3, i=1,\ldots, n$,
  \item[$3.$] $a_1, a_n=2$.
  \end{itemize}

  Для $m=1, 2, 3$ последовательностей, удовлетворяющих указанным выше условиям, не существует. Однако,
  $$
  \begin{aligned}
  &3^4=\langle2, 2, 1, 1, 1, 1, 2\rangle,\\
  &3^5=\langle2, 1, 1, 1, 3, 1, 2, 2\rangle,\\
  &3^6=\langle2, 3, 1, 3, 1, 2, 1, 1, 2\rangle,\\
  &3^7=\langle2, 1, 2, 2, 1, 2, 3, 2, 1, 2\rangle.
  \end{aligned}
  $$
  Таким образом, для $m\ge 8$ возможно применить вычисления по рекуррентной формуле:
  $$
  \begin{aligned}
  g_m=2g_{\left\lfloor\frac{m}{2}\right\rfloor}.%&g_m=2g_{\frac{m}{2}}, \text{ при } m \text { четном},\\
  %&g_m=2g_{\frac{m-1}{2}}, \text{ при } m \text { нечетном}.
  \end{aligned}
  \eqno(31)$$
  Тогда для $m\ge 8$ получаем следующую оценку:
  $$
  g_m\ge 2^{\left\lfloor log_2 \left(\frac{m+1}{8}\right)\right\rfloor}\ge \frac{m+1}{16}.
  $$
  Учитывая замечание 5, получаем утверждение теоремы.
  \end{proof}

  \begin{proof}[Доказательство теоремы 5]
  $\text{ }$\par
  Рассмотрим последовательность $(2, 1, 2)$, для которой ${\langle 2, 1, 2\rangle=2^{2^1 - 1}}$.
  Пусть ${g_{2^2-1}=1}$, для каждого ${k\ge 3}$ положим ${g_{2^k-1}=2g_{2^{k-1}-1}}$. Так как ${f_{2^2-1}\ge g_{2^2-1}}$, применяя лемму 2 с $b=2$ к указанной выше последовательности получим, что для каждого ${k\ge 3}$ ${f_{2^k-1}\ge g_{2^{k}-1}=2^{k}}$.

  \smallskip В итоге, с учетом замечания 5, получим ${f_{2^{2^k-1}}\ge2^{k}}$.
  \end{proof}

  \begin{proof}[Доказательство теоремы 6]
  $\text{ }$\par
  %\begin{itemize}
  %\item[\textit{(i)}]
  Покажем, что применением описанного метода в некоторых случаях невозможно получить более точную по порядку величины $m$ оценку, чем та, что доказана в теореме 1.
  \par(\textit{i}) Согласно описанию используемого метода $g_2=\ldots=g_{s+2}=1.$ В рекуррентную формулу для ${g_m, \text{ где } m>s+2}$ входит $\frac{s+1}{2}$ слагаемое. \\Таким образом, для $m>5s+5$ справедлива следующая оценка:
  $$
  \begin{aligned}
  &g_m=2\left(g_{\left\lfloor\frac{m}{2}\right\rfloor} + \ldots + g_{\left\lfloor\frac{m-s+1}{2}\right\rfloor}\right)\le2\frac{s+1}{2}\max_{v\le 2^{-1}m} g_v\le (s+1)^2\max_{v\le2^{-2}m} g_v \le\ldots\le\\&\le(s+1)^n\max_{v in \le 2^{-n}m} g_v \le m^{\log_2 (s+1)}.
  \end{aligned}
  $$
  При оценке было использовано, что интервалы $I_0\subset I_1\subset\ldots\subset I_n$ изменения индексов слагаемых, входящих в рекуррентную формулу для $g_m$ соответствуют (12). Также было использовано, что количество $n$ - применений формулы не может превышать $\log_2 m$.
  %\item[\textit{ii}]
  \par\medskip
  (\textit{ii})
  Рассмотрим числа $g_m$, введенные при описании метода, такие что $m=2^k-1, k\in \mathbb{N}$, а также числа ${b^i_j=\max \{\; g_l\in A_i,\quad | \quad l\in \{2^{k-1-j}-1, 2^{k-1-j}-3\}\;\}}$. Напомним, что ${A_i\text{ для } i=1, 2 \text{ определены в (21)}}.$%В этом случае
  %для вычислений нельзя подобрать матрицу размера $2\times2$, вычисления по которой дали бы
  %невозможно получить лучшую оценку, чем та, что указана
  %получена с использованием матрицы $A$ (20).
  \\При $k\ge \log_2 {5s+6}$ применим рекуррентную формулу вычисления $g_m$ (9). Получим:
  $$
  \begin{aligned}
  &g_{2^k-1}=2\left(g_{2^{k-1}-1}+g_{2^{k-1}-2}\right)= 2(1, 1)\begin{pmatrix}b^0_1\\b^0_2\end{pmatrix}=4(1, 1)\begin{pmatrix}g_{2^{k-2}-1}+g_{2^{k-2}-2}+ g_{2^{k-2}-3}\\g_{2^{k-2}-1}+g_{2^{k-2}-2}\end{pmatrix}
  \le\\
  &\le 4(1, 1)\begin{pmatrix}2\max(g_{2^{k-2}-1}, g_{2^{k-2}-3})+g_{2^{k-2}-2}\\ \max(g_{2^{k-2}-1}, g_{2^{k-2}-3})+g_{2^{k-2}-2}\end{pmatrix}=4(1, 1)A\begin{pmatrix}b^1_1\\b^1_2\end{pmatrix}\le\\
  &\le 8(1, 1)A\begin{pmatrix}2\max (g_{2^{k-3}-1}, g_{2^{k-3}-3})+g_{2^{k-3}-2}\\ \max(g_{2^{k-3}-1}, g_{2^{k-3}-3})+g_{2^{k-3}-2}\end{pmatrix}=8(1, 1)A^2\begin{pmatrix}b^2_1\\b^2_2\end{pmatrix}\le\ldots\le\\
  &\le 2^{k-3}(1, 1)A^{k-5}\begin{pmatrix}2\max (g_{2^{3}-1}, g_{2^{3}-3})+g_{2^{3}-2}\\ \max(g_{2^{3}-1}, g_{2^{3}-3})+g_{2^{3}-2}\end{pmatrix}\le 2^{k-3}(1, 1)A^{k-5}\begin{pmatrix}2g^{7} + g_6 \\g^7 + g_6\end{pmatrix}=\\
  &=2^{k-3}(1, 1)A^{k-4}\begin{pmatrix}g_6 \\g_7\end{pmatrix}=2^{k-3}(1, 1)A^{k-4}\begin{pmatrix}1 \\2\end{pmatrix}
  %\ll\max((2A)^{k-4})\le
  \ll {\lambda}^{\log_2 (m+1)}
  %\asymp(m+1)^{\log_2 {\lambda})}
  \ll m^{\log_2 {\lambda}},
  %  , 1)4\left(g_{2^{k-2}-1}+g_{2^{k-2}-2}+ g_{2^{k-2}-3}\right)\le 4(1, 1)A\begin{pmatrix}b^0_1\\b^0_2\end{pmatrix}=
  \end{aligned}
  $$
  Оценка сверху  степени матрицы получена стандартным образом с помощью приведения матрицы $A$ к жордановой форме.
  Таким образом, при вычислении $g_m$ указанного вида невозможно с помощью нашего метода получить оценку снизу лучше заявленной в теореме 1.

  (\textit{iii})
  В случае $s=2$ для $m$ - четного в рекуррентной формуле (9) для $g_m$ может быть два
  слагаемых,
   а для нечетного $m$ - одно слагаемое.
   Следовательно, для ~$m\ge 23, \text{ такого что } m=2^k-1, \text{ где } k\in\mathbb{N}$ получаем:
  $$
  \begin{aligned}
  &g_{2^k-1}=2g_{2^{k-1}-1}=4g_{2^{k-2}-1}=\ldots=2^{k-4}g_{15}=2^{k-3}=\frac{1}{8}2^{\log_2 (m+1)}=\frac{1}{8} (m+1).
  \end{aligned}
  $$
  Таким образом, утверждение теоремы доказано.

    %\end{itemize}
  \end{proof}


\newpage
\begin{thebibliography}{99}
\bibitem{hensley} Doug Hensley {\slshape Continued fractions} World Scientific Publishing Co.Pte. Ltd., 2006
\bibitem{Nie}H. Niederreiter, {\slshape Dyadic fractions with small partial quotients}, Mh. Math.
101 (1986), 309-315.
\bibitem{jap1}Monrudee Yodphotong and Vichian Laohakosol, {\slshape Proofs on Zaremba's Conjecture for Powers of 6}, Proceedings of the International Conference on Algebra and Its Applications 2002, 278-282.
\bibitem{jap2} Takao Komatsu {\slshape On a Zaremba's Conjecture For Powers}, Saraevo Journal of Mathematics, Vol.1(13). 2005, 9-13.

%\bibitem{german}О.\,Н.\,Герман, Е.\,Л.\,Лакштанов. {\slshapeО многомерном обобщении теоремы Лагранжа для цепных дробей.} Изв. РАН. Сер. матем., 2008, 72:1, 51-66
%\bibitem{moussafir}J-\,O.\,Moussafir. {\slshape Convex hulls of integral points.} Zapiski nauch. sem. POMI, 256 (2000)
%
\end{thebibliography}
\end{document}